\documentclass[11pt]{article}

\usepackage{amsmath}
\usepackage{amssymb}

\numberwithin{equation}{section}

\newtheorem{thm}{Theorem}[section]

\newtheorem{prop}[thm]{Proposition}
\newtheorem{lem}[thm]{Lemma}

\newtheorem{rem}[thm]{Remark}

\newtheorem{exam}[thm]{Example}

\def\Mod{{\mathrm {Mod}}}

\def\ga{\Gamma}

\def\R{\mathbb R}
\def\Q{\mathbb Q}
\def\Z{\mathbb Z}
\def\H{\mathbb H}

\def\T{\mathcal T}

\title{Well-rounded equivariant deformation retracts of  Teichm\"uller spaces}
\author{Lizhen Ji\thanks{Partially Supported by NSF grant DMS-1104696} 
\\ Department of Mathematics\\ University of Michigan\\ Ann Arbor, MI 48109}

\date{September 14, 2013}

\begin{document}

\maketitle

\begin{abstract}
In this paper, we construct spines, i.e.,  $\Mod_g$-equivariant
deformation retracts,  of the Teichm\"uller space $\T_g$ of compact Riemann surfaces of genus $g$.
Specifically, we define a $\Mod_g$-stable
subspace $S$ of positive codimension and construct an intrinsic
$\Mod_g$-equivariant deformation retraction from $\T_g$ to $S$.
As an essential part of the proof, we construct a canonical $\Mod_g$-deformation
 retraction of the Teichm\"uller space $\T_g$ to its thick
part $\T_g(\varepsilon)$ when $\varepsilon$ is sufficiently small.
These equivariant deformation retracts of $\T_g$ give
 cocompact models of the universal space  $\underline{E}\Mod_g$ 
for proper actions of the mapping class group $\Mod_g$. 
These deformation retractions of $\T_g$ 
are motivated by the well-rounded deformation retraction
of the space of lattices in $\R^n$. 
We also include a summary  
of results and  difficulties of an unpublished paper of Thurston on a potential spine of the Teichm\"uller space.
\end{abstract}


\section{Introduction}

Let $S_g$ be a compact oriented surface of genus $g$,
and $\Mod_g$ be the mapping class group of $S_g$.
Let $\T_g$ be the Teichm\"uller space of marked complex structures
on $S_g$.  
When $g=1$, $\T_g$ can be identified with the upper half plane $\mathbb H^2$ and $\Mod_g=\mathrm{SL}(2, \Z)$.

We will assume that $g\geq 2$ in the following.
Then every compact Riemann surface of genus $g$ admits a canonical hyperbolic metric,
and hence $\T_g$ is also the moduli space of marked hyperbolic metrics on $S_g$.

It is known that $\T_g$ is a complex manifold diffeomorphic to
$\R^{6g-6}$ and $\Mod_g$ acts holomorphically and properly on $\T_g$.
It is also known that $\Mod_g$ contains torsion elements and does not act fixed point freely on 
$\T_g$. 
By using the geodesic convexity of the Weil-Petersson metric
of $\T_g$ \cite{wo1} (or earthquakes in $\T_g$) and positive solutions of the Nielsen  
realization problem \cite{ke} \cite{wo1}, it can be shown \cite[Proposition 2.3]{jw} that $\T_g$ is a model of the 
universal space $\underline{E}\Mod_g$ of proper actions of $\Mod_g$, which means that 
for every finite subgroup $F\subset \Mod_g$, the set of fixed points $(\T_g)^F$ is nonempty
and contractible.

On the other hand, it is well-known that the quotient $\Mod_g\backslash \T_g$ is 
the moduli space of compact Riemann surfaces of genus $g$ and is non-compact.
For many applications, an important and natural problem is to
find a  model of the universal space $\underline{E}\ga$ for $\ga=\Mod_g$ 
which is $\ga$-cocompact,
i.e., the quotient $\ga\backslash \underline{E}\ga$ is  compact, or rather more
to the point, is a finite $CW$-complex. 
Another closely related problem is to find a model of $\underline{E}\ga$
which is of as small dimension as possible, for example, equal to
the virtual cohomological dimension of $\ga$.

For any $n\geq 1$, let $S_{g,n}$ be the surface obtained from $S_g$ by removing $n$ points,
and $\T_{g,n}$ be the corresponding Teichm\"uller space of $S_{g,n}$ and $\Mod_{g,n}$ the
corresponding mapping class group.
Then $\T_{g,n}$ is also a model for the universal space $\underline{E}\Mod_{g,n}$
for proper actions of $\Mod_{g,n}$. It was shown in \cite{boe} \cite{ha} \cite{pe2} that $\T_{g,n}$
admits the structure of $\Mod_{g,n}$-simplicial complex,  and hence admits
an equivariant deformation retraction to a subspace which is cofinite 
$\Mod_{g,n}$-CW-complex of dimension equal to the virtual cohomological dimension of $\Mod_{g,n}$.
This is a model of $\underline{E}\Mod_{g,n}$ of the smallest possible dimension.
This result was used by Kontsevich  \cite{ko} in proving a conjecture of Witten on intersection
theory of the moduli space $\mathcal M_{g,n}$.
The method for constructing the above spine of $\T_{g,n}$
depends crucially on the assumption that $n\geq 1$ and cannot be applied to
$\T_g$. 

Briefly, the important role played by the punctures  in triangulating the Teichm\"uller
space $\T_{g,n}$ can be explained as follows. As in \cite[Chapter 2]{ha1},  we assume that $n=1$ for simplicity.
Let $*$ be a  fixed basepoint in $S$. Then essential simple closed curves in $S_g$ passing through $*$ define a simplicial
complex $A$, called the arc-complex,   where each simplex corresponds to an arc-system of $S_g$ based at $*$,
which is a collection of essential  simple closed curves nonhomotopic to each other and intersecting only
at $*$. Let $A_\infty$ be the subcomplex consisting of
simplexes whose arc-systems do not not fill $S_g$. The basic result is that there is a canonical homeomorphism
between $\T_{g,1}$ and $A-A_\infty$. One way to see this is that for each marked Riemann surface
$(\Sigma_g, p)$ with $p$ corresponding to the basepoint $*$ in $S_g$, there is a unique, up to  multiplication by positive constants,
 horocyclic holomorphic quadratic differential 
on $\Sigma_g$ with its pole of order 2 at $p$. The foliations defined the quadratic differential will produce a filling
 arc-system together with related
weights so that they define a canonical point in $A-A_\infty$ (or rather a point in the simplex determined by the arc-system).
The second way to see this is to use the hyperbolic metric on the punctured Riemann surface $\Sigma_g-\{p\}$. 
Then a suitably defined distance of points of $\Sigma_g-\{p\}$ 
to the ideal point at infinity $p$ of $\Sigma_g-\{p\}$ defines a
spine of $\Sigma_g-\{p\}$, which also allows one to define a filling arc-system and related weights, and hence  to map
$(\Sigma_g, p)$ to a point in  $A-A_\infty$.

Once $\T_{g,1}$ is identified with $A-A_\infty$, the first barycentric subdivision of $A$ gives an equivariant spine
of $A-A_\infty$ of the optimal dimension, which gives a spine of $\T_{g,1}$ of the optimal dimension. See Remark 
\ref{explicit-spine} for more details.

As mentioned above, when $g=1$,  the Teichm\"uller space $\T_1=\mathbb H^2$, and $\Mod_1=\mathrm{SL}(2, \Z)$.
An equivariant deformation retract, i.e. a spine,  of $\mathbb H^2$ is known. 
In fact, this was used in \cite[Chapter 2]{ha1} to motivate the
above construction of the spine in $\T_{g,1}$. We will also give a construction of the spine of $\mathbb H^2$
using the identification $\mathbb H^2=\mathrm{SL}(2, \R)/\mathrm{SO}(2)$ in Remark \ref{well-rounded} below.

For the above problem to construct $\Mod_g$-cocompact universal spaces $\underline{E}\Mod_g$,
 there are two approaches 
based on the action of $\Mod_g$ on $\T_g$:
either construct a  partial compactification  $\overline{\T_g}$
such that the inclusion $\T_g\to \overline{\T_g}$
is a $\Mod_g$-equivariant homotopy equivalence, or construct a $\Mod_g$-stable subspace
$S$ such that $\Mod_g\backslash S$ is compact
and  there exists a $\Mod_g$-equivariant deformation retraction
from $\T_g$ to $S$.
The second approach seems to be more accessible and might give spaces of smaller
dimension than $\T_g$.

In a preprint \cite{th} circulated in 1985, Thurston proposed a candidate of
$\Mod_g$-equivariant deformation retract,  i.e., a spine,  of $\T_g$, of positive  codimension.
An outline was given to deform $\T_g$
into a small neighborhood of the proposed subspace.
But the deformation retraction to the proposed subspace does not necessarily achieve its goal.
See Remark \ref{thurston} below for a summary of results in \cite{th}, discussions of the difficulties, 
and an alternative proof of one key result in \cite{th}. 

It is known \cite{ha} that the virtual cohomological dimension of $\Mod_g$ is
$4g-5$. An important problem is whether there exists
a $\Mod_g$-stable subspace of $\T_g$ which is of dimension $4g-5$ and
is a $\Mod_g$-equivariant deformation retract of $\T_g$.
The question whether such a deformation retract of $\T_g$ exists or not is
Question 1.1 in \cite{bv}.  

In this paper, we consider two spines of $\T_g$. 
The first one is the thick part  $\T_g(\varepsilon)$ of $\T_g$,
i.e., for any $\varepsilon>0$ which is sufficiently small,  $\T_g(\varepsilon)$ consists of hyperbolic surfaces which do not contain geodesics with  length less than $\varepsilon$. 
The second subspace $S$ consists of hyperbolic surfaces whose systoles, i.e., 
the shortest simple
closed geodesics, contain at least an intersecting pair. (See Theorem \ref{spine} in
\S 4 for more detail).
An important point about the second spine $S$ is that it is an {\em intrinsically} defined subspace
of {\em positive} codimension. 

It is known that $\T_g(\varepsilon)$ is a real analytic submanifold with corners
and is stable under  the action of $\Mod_g$ with compact  quotient.

Existence of   a $\Mod_g$-equivariant
deformation retraction of $\T_g$  to $\T_g(\varepsilon)$
was proved in \cite[Theorems 1.2 and 1.3]{jw}.
Therefore,  $\T_g(\varepsilon)$ is a $\Mod_g$-cocompact 
$\underline{E}\ga$ space for $\ga=\Mod_g$.

On the other hand, the deformation retraction of $\T_g$  to $\T_g(\varepsilon)$ 
in \cite[\S 3]{jw} is the flow
associated with a vector field which is patched up from local vector
fields, which increase any fixed collection of short geodesics simultaneously,
 using a partition of unity. In order to get an equivariant deformation,
the construction of the partition of unity is delicate.
Since there is no intrinsic or canonical partition of unity,
the deformation retraction is not unique or  canonical.

A natural problem is to construct a deformation retraction
of $\T_g$ to $\T_g(\varepsilon)$ which depends only on the intrinsic
geometry of  the hyperbolic surfaces in $\T_g$ and the geometry of $\T_g$.
An answer is given in Theorem \ref{can-def} below.
Due to the intrinsic nature of the construction, 
it is automatically $\Mod_g$-equivariant.
Since the construction is motivated and similar to the well-rounded deformation retraction 
for the space
of lattices in $\R^n$ \cite{as1}, which is  explained in Remark \ref{well-rounded} below,
 we also call it  the {\em well-rounded deformation retraction}
of the Teichm\"uller space $\T_g$
in the title. 

The continuation of the deformation retraction to $\T_g(\varepsilon)$ gives rise
to a deformation retraction to the second spine $S$. It is a real
sub-analytic subspace of $\T_g$ 
of codimension at least 1.   See Theorem \ref{spine} below for a precise
statement.

It seems that this spine $S$
 is the {\em first example} of  equivariant deformation retract of $\T_g$
which is of {\em positive codimension}. 
A natural problem is whether this idea can possibly
be generalized to construct equivariant deformation retracts of  $\T_g$
which are of higher codimension. This will
depend on understanding subspaces of $\T_g$ consisting
of hyperbolic surfaces whose systoles intersect.
(See \cite{sch} for a survey of some work on
systoles of surfaces.)
In Proposition \ref{higher}, we explain how to obtain a   spine of $\T_g$
of codimension at least 2.

Another natural problem is to find good candidates of spines of $\T_g$ which
are of the optimal dimension $4g-5$.  It is reasonable to believe
that spines of $\T_g$ should consist of ``rounded hyperbolic surfaces",  
and hyperbolic surfaces
in the spine $S$ in Theorem \ref{spine} are in some sense the {\em least} rounded among
all possible definitions of ``rounded hyperbolic surfaces". One idea is to require hyperbolic surfaces to be
cut into smaller pieces by some systoles such that the pieces are ``rounded". 
Once good candidates are found, deformation retractions to them can be difficult.

\vspace{.1in}
\noindent {\em Acknowledgments.} I would like to thank Scott Wolpert for very helpful conversations,
references and encouragement and Hugo Parlier for helpful correspondence,
for example, the arguments in the proof of Proposition \ref{higher} are due to them,
and Juan Souto and Alexandra Pettet for the example on a flow on the unit disk
that explains a  problem with the spine in \cite{th}.
I would also like to thank an anonymous referee for constructive suggestions which have improved the
exposition of this paper. 

\section{Definition of spines and examples}

Let $X$ be a topological space, and $\ga$ a discrete group acting properly on $X$. 
A subset $S$ of $X$ is called an {\em equivariant  spine} or simply a {\em spine}
if \begin{enumerate}
\item $S$ is stable under $\ga$,
\item and there exists a $\ga$-equivariant deformation retraction from $X$ to $S$.
\end{enumerate}

For example, if $\ga$ is a cofinite, nonuniform Fuchsian group, and $X=\H^2$
is the Poincar\'e upper half-space, then $\H^2$ has an equivariant spine given by a tree.
The best known example is $\ga=SL(2, \Z)$, which is equal to $\Mod_g$ when $g=1$.
See \cite{boe}. 
(We note that when $\ga$ is a torsion-free non-uniform Fuchsian group, 
then the fundamental group $\pi_1(\ga\backslash \H^2)$ is a free group.)


\begin{prop}
If $X$ is contractible, then any spine $S$ of $X$ is also contractible.
If $X$ is a universal space for proper actions of $\ga$, then $S$ is also a universal space
for proper actions of $\ga$.
\end{prop}
\begin{proof}
We only need to note that for any finite subgroup $F$ of $\ga$, the fixed set
$S^F$ is a deformation retract of the fixed point set $X^F$ in $X$ and hence
is nonempty and contractible. 
\end{proof}

In the following we assume that $X$ is a universal space for proper actions of $\ga$.
Then $S$ is called a {\em minimal (or optimal) spine} if $\dim S=\text{vcd\ } \ga$,
where vcd$\ \ga$ is the virtual cohomological dimension of $\ga$.
The reason is that since $S$ is also a universal space for proper actions
of $\ga$, $\dim S\geq \text{vcd\ } \ga$.

The spine $S$ of $X$ is called a {\em cofinite spine} if $S$ is a $\ga$-CW complex and
the quotient $\ga\backslash S$ is a finite CW-complex.
$S$ is called a {\em cocompact spine} if $\ga\backslash S$ is a compact space.
We note that if $S$ is cofinite, then it is cocompact.
On the other hand, the converse is not automatically true,
since a general $\ga$-space may not admit the structure of
a $\ga$-CW-complex.


It is known that given any discrete group  $\ga$, there always
exists a universal space $E\ga$ for proper and fixed point free actions of $\Gamma$,
and a universal space $\underline{E}\ga$ for proper actions of $\ga$.
Both $E\ga$ and $\underline{E}\ga$
are unique up to $\ga$-equivariant homotopy equivalence
 (see \cite{lu} and references there). 
The quotient $\ga\backslash E\ga$ is a classifying space $B\ga$ of $\ga$,
i.e., $\pi_1(B\ga)=\ga$, and $\pi_i(B\ga)=\{1\}$ for $i\geq 2$.
When $\ga$ is torsion-free, then $\underline{E}\ga$ is equal to $E\ga$.
 
Models of $E\ga$  can be constructed
for groups by a general method due to Milnor  via the infinite join of copies of $\ga$ \cite{mil}
and are infinite dimensional. 
A generalization gives a construction of a model of $\underline{E}\ga$ via the infinite join of
copies of cosets $\ga/F$, where $F$ ranges over finite subgroups of $\ga$,
and such a model is also infinite dimensional. (See \cite[\S 3]{lu} and  \cite[Lemma 6.11, Chap. I]{tom}
for additional references and more details. The basic reason is that joining produces highly connected spaces,
and $\ga$ acts on these models by multiplication on the vertices,  and hence the action satisfies
the desired stabilizer property.)
But good models of $E\ga$ and $\underline{E}\ga$, in particular
those having various finiteness properties, are important in order to understand
finiteness properties of $\ga$ such as finite generation,
finite presentation and cohomological finiteness properties of $\ga$,
and also for proofs of the Novikov conjectures and
the Baum-Connes conjecture for $\ga$.
See \cite[\S 2]{jw} for an explanation of some applications. 

For some basic groups such that arithmetic subgroups (or more general discrete subgroups)
of Lie groups and mapping class groups, there are natural  finite dimensional 
$\underline{E}\ga$-spaces.
 For the former groups, they are the symmetric spaces or more general contractible
homogeneous spaces associated with the Lie groups, and for the latter groups,
they are given by the Teichm\"uller spaces.

But such natural spaces are often not  $\ga$-cofinite, or even $\ga$-cocompact,
$\underline{E}\ga$-spaces,
as pointed out in the introduction, and an important problem is to
find good equivariant spines contained in them in order to construct cofinite models of $\underline{E}\ga$-spaces of dimension as small as possible.

\begin{exam}\label{explicit-spine}
{\em 
Suppose $X$ is a simplicial complex with some faces of some simplexes  missing.
Let $X^*$ be the completion of $X$, i.e, if an open simplex is contained in $X^*$,
then all its simplicial faces are also contained in $X^*$.
Suppose $X\neq X^*$. Then there is a canonical spine of $X$ obtained as follows.
Take the maximal full subcomplex of the barycentric subdivision of $X^*$ that
are disjoint from $X^*-X$, i.e., from the missing faces of $X$. 
This is the spine constructed in \cite{as2} for the space of positive definite quadratic forms
in $n$ variables (or equivalently the space of lattices in $\R^n$), the Teichm\"uller
space $\T_{g,n}$ of Riemann surfaces of genus $g$ with $n$-punctures when $n>0$ in \cite{ha},
and the outer space  associated with the outer automorphism groups $Out(F_n)$
of the free groups in \cite{cv}. 

On the other hand, if $X$ does not have a structure of $\ga$-simplicial
complex, it is often less clear how to construct a spine or whether a spine of positive
codimension exists.
}
\end{exam}

\section{Deformation retraction of the Teichm\"uller space to the thick part and well-rounded lattices}

Let $S_g$ be a compact oriented surface of genus $g\geq 2$. 
A marked compact hyperbolic surface of genus $g$ 
 is a hyperbolic surface $\Sigma_g$ together with a homotopy
equivalence class $[\varphi]$  of diffeomorphisms $\varphi: \Sigma_g\to S_g$.
Two marked hyperbolic surfaces $(\Sigma_{g,1}, [\varphi_1]), (\Sigma_{g,2}, [\varphi_2])$
are defined to be {\em equivalent} if there exists an  isometry
$h:\Sigma_{g,1}\to \Sigma_{g,2}$
such that $[\varphi_1]=[\varphi_2\circ h]: \Sigma_{g,1}\to S_g$.
Then the Teichm\"uller space  $\T_g$ is the set of equivalence classes of marked compact hyperbolic
surfaces of genus $g$:
$$\T_g=\{(\Sigma_g, [\varphi])\}/\sim.$$
Let $\text{Diff}^+(S_g)$ be the group of orientation preserving diffeomorphisms
of $S_g$, and $\text{Diff}^0(S_g)$ the identity component of $\text{Diff}^+(S_g)$.
Then the quotient  $\text{Diff}^+(S_g)/\text{Diff}^0(S_g)$ is the {\em mapping class group}  $\Mod_d$,
and $\Mod_g$ acts on $\T_g$ by changing the markings of the marked hyperbolic
surfaces.

By the collar theorem of hyperbolic surfaces, there exists a positive constant $\varepsilon_0$ such that
for any compact hyperbolic surface $\Sigma_g$ and any two closed geodesics $\gamma_1, \gamma_2
$ in it, if 
$$\ell(\gamma_1), \ell(\gamma_2)\leq \varepsilon_0, $$
then $$\gamma_1\cap \gamma_2=\emptyset.$$

For any $\varepsilon$ with $0< \varepsilon \leq \varepsilon_0$, define the {\em 
$\varepsilon$-thick part} $\T_g(\varepsilon)$
by
$$\T_g(\varepsilon)=\{(\Sigma_g, [\varphi])\mid \text{for all simple closed geodesic $\gamma$ in
$\Sigma_g$}, \ \ell(\gamma)\geq \varepsilon\}.$$ 

Then the following result is known.

\begin{prop}\label{truncated-subspace}
The subspace  $\T_g(\varepsilon)$ is stable under $\Mod_g$ with a compact quotient $\Mod_g\backslash
\T_g(\varepsilon)$. 
Under the above assumption that $0< \varepsilon\leq \varepsilon_0$, 
$\T_g(\varepsilon)$ is a real analytic manifold with corners and hence admits a
$\Mod_g$-equivariant triangulation such that $\Mod_g\backslash \T_g(\varepsilon)$
is a finite CW-complex.
\end{prop}
\begin{proof}
It is clear that $\T_g(\varepsilon)$ is stable under $\Mod_g$ since its definition
does not depend on the markings.
The compactness of the quotient $\Mod_g\backslash
\T_g(\varepsilon)$ follows from the Mumford compactness criterion for subsets of 
$\Mod_g\backslash \T_g$ \cite{mu}. 
Near any boundary point $p=(\Sigma_g, [\varphi])\in \partial \T_g(\varepsilon) - \T_g(\varepsilon)$,
the subspace $\T_g(\varepsilon)$ is defined by the inequalities:
$$\ell(\gamma_1), \cdots, \ell(\gamma_k)\geq \varepsilon,$$
where $\gamma_1, \cdots, \gamma_k$ are all the simple closed geodesics
on the marked surface $\Sigma_g$ such that 
$\ell(\gamma_1)(p)=\varepsilon, \cdots, \ell(\gamma_k)(p)=\varepsilon$.
By the assumption on $\varepsilon$, the geodesics 
$\gamma_1, \cdots, \gamma_k$ are disjoint. 
Then they can form a part of a collection of a pants decomposition of $\Sigma_g$,
and their length functions are a part of the associated Fenchel-Nielsen
coordinates, 
and hence their differentials $d \ell(\gamma_1), \cdots,  d \ell(\gamma_k)$
are linearly independent. 
This implies that 
the subspace of $\T_g$ defined by the inequalities $\ell(\gamma_1) \geq \varepsilon, \cdots, \ell(\gamma_k)\geq\varepsilon$
is a real analytic submanifold of $\T_g$ with corners near the point $(\Sigma_g, [\varphi])$. 
\end{proof}

The main result of \cite[Theorems 1.2 and 1.3]{jw} is the following result.

\begin{prop} For every $\varepsilon$ with $0<\varepsilon\leq \varepsilon_0$, 
there exists a $\Mod_g$-equivariant deformation retraction of $\T_g$ to $\T_g(\varepsilon)$.
In particular, $\T_g(\varepsilon)$ is a cofinite model of the 
universal space $\underline{E}\ga$ for proper actions
of $\ga=\Mod_g$. 
\end{prop}

The idea of the proof in \cite{jw} is as follows.
For any marked hyperbolic surface $(\Sigma_g, [\varphi])$ in the thin part $\T_g-\T_g(\varepsilon)$,
we increase the lengths of the short geodesics by following the  flow
of a local vector field which is a suitable linear combination of the gradient vectors
of the length functions of these short geodesics.
Specifically, let $\gamma_1, \cdots, \gamma_k$ be all the short geodesics
of $\Sigma_g$ such that $\ell(\gamma_1)\leq  \cdots\leq \ell(\gamma_k) \leq \varepsilon$.
In a simple case, suppose that $\ell(\gamma_1)<\ell(\gamma_2)$.
For any geodesic $\gamma$, let $\nabla \ell(\gamma)$ be the gradient of the function
$\ell(\gamma)$ with respect to the Weil-Petersson metric of $\T_g$.
Then the flow along the vector field
$\nabla \ell(\gamma_1)$ will increase $\ell(\gamma_1)$ until it reaches
$\ell(\gamma_2)$ or $\varepsilon$.
The point to make use of the Weil-Petersson metric is that it is intrinsic
and hence the flow is automatically $\Mod_g$-equivariant. 

On the other hand, a difficulty occurs if $\ell(\gamma_1)=\ell(\gamma_2)$ since it is not clear
whether we should use either $\nabla \ell(\gamma_1)$ or $\nabla \ell(\gamma_2)$.

The way to solve this problem in \cite{jw} consists of two steps:
(1) introduce a local vector field near every point in the thin part $\T_g-\T_g(\varepsilon)$, 
which, in the notation above,  is a suitable
linear combination of  $\nabla \ell(\gamma_1), \cdots,
\nabla \ell(\gamma_k)$ on a small neighborhood
of  $(\Sigma_g, [\varphi])$ in $\T_g$ such that under its flow, 
the lengths of all the short geodesics $\gamma_1, \cdots, \gamma_k$ 
are increased  simultaneously, 
(2) use a suitable $\Mod_g$-invariant partition of unity to glue up the
local vector fields to obtain a desired global vector field on the thin part
$\T_g-\T_g(\varepsilon)$ that is {\em invariant} under $\Mod_g$.

The construction of the partition of unity in Step (2) is complicated, but not
canonical or intrinsic.
A natural problem is to obtain an equivariant deformation retraction of
$\T_g$ to $\T_g(\varepsilon)$ which only depends on the intrinsic geometry
of the hyperbolic surfaces $\Sigma_g$ and the geometry of $\T_g$.
The first purpose of this paper is to construct such an intrinsic equivariant
deformation retraction. To do this, we first recall the well-rounded deformation retraction
of lattices in $\R^n$.

\begin{rem}\label{well-rounded}
{\em {\bf \  Well-rounded deformation retraction of lattices.}
The pair $(\T_g, \Mod_g)$ has often be compared with the pair 
$(SL(n, \R)/SO(n), SL(n, \Z))$ of 
a symmetric space $SL(n, \R)/SO(n)$ of noncompact type and an arithmetic subgroup 
$SL(n, \Z)$ acting on it.
Unlike a general symmetric space, $SL(n, \R)/SO(n)$ is the moduli space of marked unimodular
lattices in $\R^n$, where a {\em marked lattice} is a lattice $\Lambda\subset \R^n$
together with an ordered basis $v_1, \cdots, v_n$ of $\Lambda$,
and a lattice $\Lambda\subset \R^n$ is {\em unimodular} if  $vol(\Lambda\backslash \R^n)=1$. 
The locally symmetric space $SL(n, \Z)\backslash SL(n, \R)/SO(n)$ is the moduli space
of unimodular lattices in $\R^n$ up to isometry.
As pointed out above, when $n=2$, $SL(n, \R)/SO(n)$  is equal to the Teichm\"uller space $\T_1$,
and the deformation retraction described below gives an equivariant spine of $\T_1$. 

There is a known $SL(n, \Z)$-equivariant deformation retraction of $SL(n, \R)/SO(n)$
to the subspace of well-rounded lattices by successively scaling up the inner product
on the linear subspace spanned by the shortest lattice vectors and hence
increasing the length of shortest geodesics of the associated flat tori in order to reach 
more rounded flat tori. According to \cite{as1}, this result is due to Soule and Lannes
and  was presented 
in the unpublished thesis of Soule. A generalization of this method to the symmetric space 
associated with the general linear group of a division algebra
over $\Q$ is given in \cite{as1}. Since we only need the above special case 
$(SL(n, \R)/SO(n), SL(n, \Z))$, 
we give a simplified summary  of the deformation retraction in \cite{as1}
to motivate the deformation retraction of $\T_g$ in the next section. 

More precisely, let $\R^n$ be given the usual Euclidean inner
product $\langle, \rangle$, and $\Lambda\subset \R^n$ be a lattice.
Let 
$$m(\Lambda)=\inf \{\langle v, v\rangle \mid v\in \Lambda-0\},$$
and $$M(\Lambda)=\{v\in \Lambda\mid \langle v, v\rangle = m(\Lambda)\}.$$
If $M(\Lambda)$ spans $\R^n$, then the lattice $\Lambda$ is called a {\em well-rounded lattice}.
If $\Lambda\subset \R^n$ is a unimodular, not well-rounded lattice in $\R^n$,
then it can be deformed {\em canonically}
 to a well-rounded unimodular lattice in several steps.

These notions can also be defined for marked lattices.
Since there is a natural marking in the deformation, we suppress the marking in the following
discussion. Or equivariantly,  we are defining a deformation retraction
of the locally symmetric space $SL(n, \Z)\backslash SL(n, \R)/SO(n)$.

Suppose $\Lambda$ is not a well-rounded lattice. Then the span $M(\Lambda)\otimes \R$,
denoted by $V_M(\Lambda)$, 
is a proper linear subspace of $\R^n$. Let  $V_M^\perp(\Lambda)$ be the orthogonal
complement of $V_M(\Lambda)$ in $\R^n$. 
For any $t\geq 1$, define a new inner product $\langle, \rangle_t$
 on $\R^n$ such that on $V_M(\Lambda)$,
the inner product is scaled up by $t$, and on $V_M^\perp(\Lambda)$, it is scaled down
by a unique factor so that with respect to the new inner $\langle, \rangle_t$ on $\R^n$,
the lattice $\Lambda$ is still a unimodular lattice. Note that this inner product depends
on the lattice $\Lambda$. Scaling
on the subspaces $V_M(\Lambda)$ and $V_M^\perp(\Lambda)$
in the opposite direction gives a canonical isometric identification
between the two Euclidean spaces $(\R^n, \langle, \rangle_t)$ and $(\R^n, \langle, \rangle)$,
and the image of $\Lambda$ gives a new lattice $\Lambda_t$ in the standard 
Euclidean space $(\R^n, \langle, \rangle)$. 

If $\Lambda$ is given a marking, i.e., an ordered basis $v_1, \cdots, v_n$, then
the images of $v_1, \cdots, v_n$  in $\Lambda_t$ form a basis of $\Lambda_t$ as well.
 What is changed is the inner
product and hence  lengths of the vectors. This process
 also gives a canonical identification between
marked lattices $\Lambda_t$ and $\Lambda$.

In the deformation $\Lambda_t$, the monimal norm $m(\Lambda_t)$ is increasing.
We deform $\Lambda$ to $\Lambda_t$ until $M(\Lambda_t)$ contains at least one more
independent lattice vector, i.e., the dimension of $M(\Lambda_t)\otimes \R$
increases at least by 1 (note that for small values of $t$, $M(\Lambda_t)$ stays
constant under the above identification between $\Lambda_t$ and $\Lambda$).
 This finishes the first step of the deformation. Now we start again
using the new vector subspace $V_M(\Lambda_t)$, and deform it again by increasing
the norms of all {\em minimal lattice vectors simultaneously at the same rate}.
After finitely many steps, $M(\Lambda_t)$ spans $\R^n$ and the lattice $\Lambda_t$
is well-rounded.

Though the above deformation procedure is canonical, we still need to show that it
gives a {\em continuous} map on $Y=SL(n, \Z)\backslash SL(n, \R)/SO(n)$. 
To do this, we define a filtration of $Y$: 
$$Y=Y_1\supset Y_2\cdots \supset Y_n, $$
where $Y_j=\{\Lambda \in Y\mid \dim M(\Lambda)\otimes \R\geq j\}$, $j=1, \cdots, n$.
Clearly, $Y_1=Y$, and $Y_n$ is the subspace of well-rounded lattices.
The complement $Y_{j-1}-Y_{i}$ consists of lattices whose associated subspace $M(\Lambda)\otimes \R$ has dimension equal to $j-1$.   
The deformation retraction to  the subspace $Y_n$ of well-rounded lattices consists of 
composition of the deformation
retractions of $Y_{j-1}$ to $Y_{i}$ for $j=2, \cdots, n$.
Hence it suffices to show that  at every step, the retraction of $Y_{j-1}$ to $Y_{j}$ is continuous.

Let $\Lambda, \Lambda'\in Y_{j-1}$ be two lattices. If $\Lambda\in Y_{j-1}-Y_j$,
we deform it to $\Lambda_t\in Y_j$ by the procedure described above; otherwise, $\Lambda\in Y_j$ and set $\Lambda_t=\Lambda$. Similarly, we can define the deformation image
$\Lambda'_{t'}\in Y_j$.
We need to show that  $\Lambda_t$ and $\Lambda'_{t'}$ are close  whenever
$\Lambda, \Lambda'$ are close.

There are two cases:
(1) Suppose $\Lambda\in Y_j$. If $\Lambda'\in Y_j$, then $\Lambda'_{t'}=\Lambda'$
is close to $\Lambda_t=\Lambda$ by assumption.
If $\Lambda'\in Y_{j-1}-Y_j$, then the next shortest norm of vectors in $\Lambda'$
after $m(\Lambda')$ is close to $m(\Lambda')$.
The reason is that $\dim M(\Lambda)\otimes \R$ is 
at least $j$ but $\dim M(\Lambda')\otimes \R=j-1$.
This implies that the stretching factor
in reaching $\Lambda'_{t'}$ from $\Lambda'$ is close to 1, and  $\Lambda'_{t'}$ is close to $\Lambda'$ and hence close to $\Lambda_t=\Lambda$.

(2) Suppose that $\Lambda\in Y_{j-1}-Y_j$. Then $\dim M(\Lambda)\otimes \R=j-1$.
We claim that when $\Lambda'$ is close enough to $\Lambda$,
then $\dim M(\Lambda')\otimes \R=j-1$.  
By assumption, for any $v\in \Lambda-M(\lambda)-\{0\}$, $||v||> m(\Lambda)$,
and hence these exists a positive number $\varepsilon$ such that 
for all $v\in \Lambda-M(\lambda)-\{0\}$, $||v||\geq  m(\Lambda)+\varepsilon.$
Then for any $\Lambda'\in Y_{j-1}$ close to $\Lambda$, there exists $g\in SL(n, \R)$
close to the identity element such that $\Lambda'=g \Lambda$. This implies that only
minimal vectors in $M(\Lambda)$ can be mapped to $M(\Lambda')$,
i.e., $M(g\Lambda)\subseteq g M(\Lambda)$, and hence
$\dim M(\Lambda')\otimes \R\leq \dim M(\Lambda)\otimes \R$.
Since $\Lambda'\in Y_{j-1}$, $\dim M(\Lambda')\otimes \R \geq j-1$.
By assumption, $\dim M(\Lambda)\otimes \R=j-1$, it follows that
$\dim M(\Lambda')\otimes \R=\dim M(\Lambda)\otimes \R$.

Since $m(\Lambda)$ and $m(\Lambda')$ are close and the next smallest norms in $\Lambda, \Lambda'$ are also close, the equality of the dimension 
$\dim M(\Lambda')\otimes \R=\dim M(\Lambda)\otimes \R$ implies that 
 the scaling factor $t$ needed for $M(\Lambda_t)$ to have a higher dimension,
i.e., for $\Lambda_t$ to reach $Y_j$ is close
to the scaling factor $t'$ needed for $\Lambda'_{t'}$ to reach $Y_j$. 
This implies that $\Lambda_t$ and $\Lambda'_{t'}$ are close. This completes
the proof of the continuity of the deformation retraction from $Y_{j-1}$ to $Y_j$, and hence
of the deformation retraction of $Y$ to the subspace $Y_n$ of well-rounded lattices.
}
\end{rem}

\vspace{.1in}

A tempting idea is to carry out the same deformation for Teichm\"uller spaces
by increasing the lengths of shortest geodesics while decreasing the lengths
of other geodesics. But there are no linear structures and orthogonal complement
on the set of closed geodesics
as in the case of lattices in $\R^n$, and it is not clear whether such a deformation is 
possible. We will need to deform differently. 

For any  positive $\varepsilon\leq \varepsilon_0$, decompose the thin part
$\T_g-\T_g(\varepsilon)$ into a disjoint union of submanifolds
according to the multiplicity of the shortest geodesics, or {\em systoles}.

Every simple closed curve $c$ of the base surface $S_g$
which is not homotopic to a point 
induces a unique simple closed geodesic in every marked
hyperbolic surface $(\Sigma_g, [\varphi])$, which is contained 
in the homotopy class $[\varphi^{-1}(c)]$ of simple closed curves. 

For a collection of pairwise disjoint simple closed curves $c_1, 
\cdots, c_k$ of $S_g$, let $\gamma_1, \cdots, \gamma_k$ be the corresponding
geodesics in $(\Sigma_g, [\varphi])$.
Define a subspace 
$$\T_{g, c_1, \cdots, c_k}=\{(\Sigma_g, [\varphi])\mid \ell(\gamma_1)=\cdots=\ell(\gamma_k)< \ell(\gamma),
\text{\ for any other simple closed geodesic\ } \gamma\subset \Sigma\}.$$ 
In terms of systoles, $\T_{g, c_1, \cdots, c_k}$ consists of
hyperbolic surfaces whose systoles  are disjoint simple closed geodesics
$\gamma_1, \cdots, \gamma_k$.

\begin{prop}\label{decomposition}
For every collection of disjoint simple closed curves $c_1, \cdots, c_k$,   the index
$k$ satisfies the bound: $k\leq 3g-3$.
The intersection $(\T_g-\T_g(\varepsilon))\cap \T_{g, c_1, \cdots, c_k}$ is 
a nonempty  real analytic submanifold.
The thin part $\T_g-\T_g(\varepsilon)$ admits a $\Mod_g$-equivariant
{\bf disjoint decomposition} into $(\T_g-\T_g(\varepsilon))\cap \T_{g, c_1, \cdots, c_k}$,
when $\{c_1, \cdots, c_k\}$ ranges over all possible collections of disjoint simple closed
curves of the base surface $S_g$. 
\end{prop}
\begin{proof}
The first statement is standard that the maximum number of disjoint, simple closed geodesics
  in every hyperbolic surface $\Sigma_g$ is equal to $3g-3$.
The second  statement follows from the proof of  Proposition \ref{truncated-subspace}.
and  the fact that for any collection of disjoint simple closed curves $c_1, \cdots, c_k$,
 there are marked hyperbolic surfaces whose corresponding simple closed geodesics
 have arbitrarily short lengths.
 
 The third statement follows from the fact that for any hyperbolic surface
 $\Sigma_g$, the lengths of its simple closed
 geodesics form an increasing sequence with finite multiplicity going to infinity,
and hence the surface $\Sigma_g$
 belongs to some subspace $\T_{g, c_1, \cdots, c_k}$, where $c_1, \cdots, c_k$
are disjoint since their lengths are less than $\varepsilon$.
\end{proof}

 For each collection $c_1, \cdots, c_k$ of disjoint, simple closed curves of $S_g$,
 we define a vector field $V_{c_1, \cdots, c_k}$
 on  the associated subspace $(\T_g-\T_g(\varepsilon))\cap \T_{g, c_1, \cdots, c_k}$ as follows.
 Since $(\T_g-\T_g(\varepsilon))\cap \T_{g, c_1, \cdots, c_k}$ is a submanifold
 of $\T_g$, the Weil-Petersson metric of $\T_g$ restricts to a Riemannian metric
 on it.
By definition, the length functions $\ell(\gamma_1), \cdots, \ell(\gamma_k)$
are equal on $ \T_{g, c_1, \cdots, c_k}$ and hence define a common
function, denoted by $\ell$.
Let $\nabla \ell^{1/2}$ be the gradient of $\ell^{1/2}$ with respect to the
restricted Riemannian metric on  $(\T_g-\T_g(\varepsilon))\cap \T_{g, c_1, \cdots, c_k}$.

\begin{lem}\label{bounded}
In the above notation,  the function $\ell$ has no critical point
on $(\T_g-\T_g(\varepsilon))\cap \T_{g, c_1, \cdots, c_k}$ and 
the vector field  $\nabla \ell^{1/2}$, denoted by $V_{c_1, \cdots, c_k}$,
 does not vanish at any point
on $(\T_g-\T_g(\varepsilon))\cap \T_{g, c_1, \cdots, c_k}$. 
Furthermore, $\nabla\ell^{1/2}$
is uniformly bounded away from 0 and from above.
\end{lem}
\begin{proof}
Since the geodesics $\gamma_1, \cdots, \gamma_k$ are disjoint, their length functions
$\ell_1, \cdots, \ell_k$
appear as a part of the Fenchel-Nielsen coordinates associated
with a collection of maximal disjoint simple closed geodesics. This implies that
$\nabla\ell_i\neq 0$ for $i=1, \cdots, k$.
By \cite[Lemma 3.12]{wo3}, for any two disjoint geodesics $\gamma_i, \gamma_j$, 
$\langle \nabla\ell_i, \nabla\ell_j\rangle >0$.
This implies that on $\T_{g, c_1, \cdots, c_k}$, which can be thought of a partial diagonal,   $\nabla\ell\neq 0$ at every point.
This implies that $\nabla\ell^{1/2} \neq 0$ too. 
When $\ell$ is small, the uniform boundedness 
of $\nabla\ell^{1/2}$ follows from \cite[Lemma 3.12]{wo3} 
(also \cite[equation (3.1)]{jw}).
\end{proof}

\begin{rem}
{\em
  We note that an important reason for using $\nabla\ell^{1/2}$
instead of $\nabla \ell$ is that the former is uniformly bounded away
from 0 and from above when $\ell$ belongs to $(0, a]$ for any $a>0$, in particular near 0.
 See the discussion on \cite[page 278]{wo2}.
}\end{rem}

\begin{lem}\label{vector-field}
The vector fields $V_{c_1, \cdots, c_k}$ together define a $\Mod_g$-equivariant
vector field  on the thin part $\T_g-\T_g(\varepsilon)$. Denote this vector field by $V$,
\end{lem}  
 \begin{proof}
By Proposition \ref{decomposition}, the thin part $\T_g-\T_g(\varepsilon)$ admits
a $\Mod_g$-equivariant {\em disjoint} decomposition into $(\T_g-\T_g(\varepsilon))\cap \T_{g, c_1, \cdots, c_k}$.
Therefore, the vector fields $V_{c_1, \cdots, c_k}$ combine and define a vector
field on $\T_g-\T_g(\varepsilon)$.
Since the submanifolds and the vector fields are defined intrinsically in terms
of the length functions and the Weil-Petersson metric,
it is clear that $V$ is equivariant with respect to $\Mod_g$.
 \end{proof}
 
\begin{rem}
{\em 
We note that the vector field $V$ is {\em not continuous} in general.
For example, $\T_{c_1,  c_2}$ is contained in the closure of 
both $\T_{c_1}$ and $\T_{c_2}$. The vector fields $V_{c_1}$ and $V_{c_2}$
will both extend continuously to $\T_{c_1,  c_2}$ but have different values
at these boundary points. The vector field $V_{c_1, c_2}$ is an average of
$V_{c_1}$ and $V_{c_2}$.
The same pheonmenon of discontinuity occurs in the well-rounded deformation of lattices
in $SL(n, \R)/SO(n)$ in Remark \ref{well-rounded}. But the deformation paths are continuous.
In some sense, it amounts to the fact that the integral of a piecewise continuous function
is continuous.}
\end{rem}

\begin{thm}\label{can-def}
 For every $\varepsilon$ with $0<\varepsilon\leq \varepsilon_0$, 
there exists an intrinsic $\Mod_g$-equivariant deformation retraction of $\T_g$ to $\T_g(\varepsilon)$.
In particular, $\T_g(\varepsilon)$ is a cofinite model of the 
universal space $\underline{E}\ga$ for proper actions
of $\ga=\Mod_g$. 
\end{thm}
\begin{proof}
The deformation retraction is the flow  of the vector field $V$ defined in Lemma \ref{vector-field}.
Though the vector field $V$ is not continuous, it does not cause any problem and the flow
is still continuous. Roughly, for any hyperbolic surface in the thin part $\T_g-\T_g(\varepsilon)$,
we increase the lengths of the systoles at the same rate
until we have reached the thick part $\T_g(\varepsilon)$ or the systolic length equals
to the length of the next shortest geodesic, i.e, the systoles have included more geodesics, 
and then we repeat the above procedure. It will reach the thick part and
stop after finitely many steps.

Specifically, we  deform any marked hyperbolic surface
$(\Sigma_g, [\varphi])$ in the thin part $\T_g-\T_g(\varepsilon)$ as follows.
Let $\ell(\gamma_1)\leq \ell(\gamma_2)\leq \cdots\ell(\gamma_n)\leq \cdots $ be the lengths of  
its simple closed geodesics arranged in the increasing order.

Suppose $\ell(\gamma_1)= \ell(\gamma_2)= \cdots=\ell(\gamma_k) <\ell(\gamma_{k+1})$,
i.e.,  the systoles consist of $\gamma_1, \cdots, \gamma_k$.
This means that $(\Sigma_g, [\varphi])\in \T_{g, c_1, \cdots, c_k}.$
Then we increase the lengths $\ell(\gamma_1), \cdots, \ell(\gamma_k)$ simultaneously
at the same rate, i.e., we keep the deformation path inside $\T_{g, c_1, \cdots, c_k}$.
This can be achieved by the flow 
of the nowhere vanishing vector field $V_{c_1, \cdots, c_k}$ on the submanifold 
$(\T_g-\T_g(\varepsilon))\cap \T_{g, c_1, \cdots, c_k}$
until the length of the systoles reach the next length $\ell(\gamma_{k+1})$ or the value $\varepsilon$,
i.e., the surface has reached a point of the thick part $\T_g(\varepsilon)$.

Suppose that the deformation has not reached the thick part $\T_g(\varepsilon)$ yet. 
At the next step, we have $ \ell(\gamma_1)= \ell(\gamma_2)= \cdots=\ell(\gamma_{k'}) <
\ell(\gamma_{k'+1})$, where $k'\geq k+1$. Since $\ell(\gamma_1)= \ell(\gamma_2)= \cdots=\ell(\gamma_{k'})<\varepsilon$, $\gamma_1, \cdots, \gamma_{k'}$ are disjoint. 
We can deform as above using the vector field $V_{c_1, \cdots, c_{k'}}$ in order to increase
the  length of the systoles at the same rate and stop if the systolic length reaches
the length of the next shortest geodesic or the surface has reached a point
of the thick part.
Since we have reached $\T_g(\varepsilon)$ already whenever $k'> 3g-3$,
this process will terminate after at most $3g-3$ steps.

To show that the deformation retraction is  continuous,  we follow the proof of the continuity
of the well-rounded deformation retraction of lattices in $\R^n$ recalled
in Remark \ref{well-rounded}.

For $j=1, \cdots, 3g-3$, let $\T_{g}^j$ be the subspace of $\T_{g}$ of marked hyperbolic
surfaces whose systoles contain at least $j$ disjoint simple geodesics.
Then $\T_{g}^1=\T_g$. Define $\T_g^{3g-2}=\emptyset$.
For each $j=2, \cdots,  3g-2$,
the above discussion gives a deformation retraction of $\T_g^{j-1}\cup \T_g(\varepsilon)$
to  $\T_g^{j}\cup \T_g(\varepsilon)$,
and  their composition gives the deformation retraction 
to the thick part $\T_g(\varepsilon)$. 
Therefore, it suffices to show that the deformation retraction at each step,
from $\T_g^{j-1}\cup \T_g(\varepsilon)$
to  $\T_g^{j}\cup \T_g(\varepsilon)$, is continuous. 

Fix a $j\in {1, \cdots, 3g-2}$.
Let $(\Sigma_g, [\varphi]), (\Sigma_g', [\varphi]')$ be two points in $\T_g^{j-1}\cup \T_g(\varepsilon)$. Denote their deformed image in $\T_g^j\cup \T_g(\varepsilon)$
by $(\Sigma_g, [\varphi])_t$ and $(\Sigma_g', [\varphi]')_{t'}$. (The subscripts indicate the times
needed for the flow). 
We need to show that when $(\Sigma_g, [\varphi]), (\Sigma_g', [\varphi]')$ are close,
then $(\Sigma_g, [\varphi])_t$,  $(\Sigma_g', [\varphi]')_{t'}$ are also close. 
There are two cases to consider.

\vspace{.05in}
{\em Case} (1): $(\Sigma_g, [\varphi]) \in \T_g^j\cup \T_g(\varepsilon)$. 
Then $(\Sigma_g, [\varphi])_t=(\Sigma_g, [\varphi])$.
If $(\Sigma_g', [\varphi']) \in \T_g^j\cup \T_g(\varepsilon)$,
then $(\Sigma_g', [\varphi'])_{t'}=(\Sigma_g', [\varphi'])$ is close
to $(\Sigma_g, [\varphi])_t$.

Otherwise, $(\Sigma_g', [\varphi']) \in (\T_g^{j-1} -\T_g^j)\cap \T_g-\T_g(\varepsilon)$,
and the systoles of $(\Sigma_g', [\varphi'])$ consists of $j-1$ geodesics.
Since $(\Sigma_g', [\varphi]')$ is close to $(\Sigma_g, [\varphi])$, and the systoles
of $(\Sigma_g, [\varphi])$ consists of at least $j$ geodesics, it imples that
the  length of $j$th shortest geodesic of $(\Sigma_g', [\varphi]')$ is close to its systolic
length, and it takes a small deformation for $(\Sigma_g', [\varphi'])_{t'}$ to reach a point
in $\T_g^j\cup \T_g(\varepsilon)$. (Here we have used the fact that 
by Lemma \ref{bounded},
  each vector field $V_{c_1, \cdots, c_k}$ is continuous and its norm is uniformly
bounded from both below and above.)
Therefore, 
$(\Sigma_g, [\varphi])_t$,  $(\Sigma_g', [\varphi]')_{t'}$ are also close. 

\vspace{.05in}
{\em
Case} (2): $(\Sigma_g, [\varphi]) \in (\T_g^{j-1} -\T_g^j)\cap \T_g-\T_g(\varepsilon)$.
Then the systoles of $(\Sigma_g, [\varphi])$ consist of $j-1$ geodesics,
$\gamma_1, \cdots, \gamma_{j-1}$, which correspond to simple closed
curves $c_1, \cdots, c_{j-1}$ of the base surface $S_g$.
We claim that when $(\Sigma_g', [\varphi'])$ is sufficiently close to
$(\Sigma_g, [\varphi])$, then the systoles of $(\Sigma_g', [\varphi'])$ also consist
of $j-1$ geodesics $\gamma_1', \cdots, \gamma_{j-1}'$
 which correspond to the same set of simple closed curves
$c_1, \cdots, c_{j-1}$ of the base surface $S_g$. 
To prove the claim, we note that there exists a positive number $\delta$ such
that for every simple closed geodesic $\gamma$ of
$(\Sigma_g, [\varphi])$ different from $\gamma_1, \cdots, \gamma_{j-1}$,
$\ell(\gamma)\geq \ell(\gamma_1)+\delta$.
This implies that when $(\Sigma_g', [\varphi'])$ is sufficiently close to
$(\Sigma_g, [\varphi])$, only geodesics of $(\Sigma_g', [\varphi'])$ correspond to
the simple closed curves $c_1, \cdots, c_{j-1}$ on the base surface $S_g$ can be systoles. 
Since $(\Sigma_g', [\varphi'])\in \T_g^{j-1}$, it must have at least $j-1$ systoles.
This implies that it has exactly $j-1$ systoles and the claim is proved.

By the claim, $(\Sigma_g, [\varphi])$, $(\Sigma_g', [\varphi'])$ belong
to the same submanifold $\T_{g, c_1, \cdots, c_k}$. 
Then under the flow defined by the vector field $V_{c_1, \cdots, c_k}$, 
they reach their deformation points $(\Sigma_g, [\varphi])_t$, $(\Sigma_g', [\varphi'])_{t'}$
in $\T_g^j\cup T_g(\varepsilon)$.
We note that by Lemma \ref{bounded},
  each vector field $V_{c_1, \cdots, c_k}$ is continuous and its norm is uniformly
bounded from both below and above, and hence the time it takes to move any  point
of $(\T_g^{j-1}-\T_g(\varepsilon))\cap \T_{g, c_1, \cdots, c_k}$
to  $(\T_g^{j}\cup\T_g(\varepsilon))$
is also uniformly bounded in terms of its distance to  $(\T_g^{j}\cup\T_g(\varepsilon))$
and depends continuously on the initial point.
Therefore, when $(\Sigma_g, [\varphi])$, $(\Sigma_g', [\varphi'])$ are close,
$(\Sigma_g, [\varphi])_t$, $(\Sigma_g', [\varphi'])_{t'}$ are also close,
and the continuity of the deformation retraction of $\T_g^{j-1}$ to $\T_g^j$ is proved. 

Since the flow at every step and hence the whole flow from $\T_g$ to $\T_g(\varepsilon)$
is intrinsically defined and hence equivariant with respect to $\Mod_g$, 
this completes the proof of Theorem \ref{can-def}.
\end{proof}

\section{Deformation of  Teichm\"uller space to a spine of positive codimension}

Though the thick part $\T_g(\varepsilon)$ of $\T_g$ gives a cofinite model of the
universal space 
$\underline{E}\ga$ for $\ga=\Mod_g$, it is a subspace of $\T_g$ of codimension 0.

It is tempting to conjecture that $\T_g$ admits  a $\Mod_g$-equivariant 
deformation retraction to a subspace of dimension equal to $4g-5$,
which is equal to the virtual cohomological dimension of $\Mod_g$.
 
One modest step towards this is to construct subspaces of $\T_g$ which are
of positive codimension and equivariant deformation retracts of $\T_g$
to them.
As mentioned in the introduction,
such an attempt was first made in \cite{th}. 
In this section, we continue the flow in the previous section and
deform $\T_g(\varepsilon)$ (or rather $\T_g$) to a subspace of positive codimension. 

For any marked hyperbolic surface $(\Sigma_g, [\varphi])$, arrange
its lengths of simple closed geodesics in the increasing order:
$$\ell(\gamma_1)\leq \ell(\gamma_2) \leq \cdots \leq \ell(\gamma_n)\leq \cdots.$$

Define a  {\em well-rounded subspace }
$S\subset \T_g$ to consist of marked hyperbolic surfaces $(\Sigma_g, [\varphi])$
satisfying the conditions: 
\begin{enumerate}
\item $\ell(\gamma_1)=\cdots=\ell(\gamma_k)<\ell(\gamma_{k+1})$ for some $k\geq 2$.
\item some pairs of geodesics from $\gamma_1, \cdots, \gamma_k$, i.e., some pairs
of systoles of $(\Sigma_g, [\varphi])$, intersect each other.
\end{enumerate}

\begin{prop}
The well-rounded subspace $S$ is stable under $\Mod_g$ with a compact quotient,
and the codimension of $S$ in $\T_g$ is positive.
Furthermore, $S$ is a sub-analytic subspace and hence admits a $\Mod_g$-equivariant
triangulation such that the quotient $\Mod_g\backslash S$ is a finite CW-complex. 
\end{prop}
\begin{proof}
It is clear that $S$ is stable under $\Mod_g$ since it is defined in terms of the lengths of closed
geodesics of hyperbolic surfaces in $\T_g$. 
For any hyperbolic surface in $S$, since two of the shortest geodesics intersect,
by the collar theorem, their length is uniformly bounded from below by a constant
which depends only on $g$. 
Then by the Mumford compactness criterion \cite{mu}, the quotient $\Mod_g\backslash S$ is compact.

Near any point in $S$, $S$ is locally defined by at least one real analytic equation,
$\ell(\gamma_1)=\ell(\gamma_2), \cdots, \ell(\gamma_{k-1})=\ell(\gamma_k)$.
This implies that $S$ is of positive codimension. 

Since the geodesic length functions are real analytic, $S$ is a sub-analytic space.
The existence of equivariant triangulation of $S$ follows from a general result on existence
of equivariant triangulation.  
\end{proof}

The second result of this paper is to show that $S$ is an equivariant
deformation retraction of $\T_g$.

\begin{thm}\label{spine}
The well-rounded subspace $S$ is a cofinite, equivariant spine of $\T_g$ of positive codimension
with respect to $\Mod_g$.
\end{thm}
\begin{proof}
For any collection of simple closed curves $c_1, \cdots, c_k$
of the base surface $S_g$, by the same proof of Proposition \ref{decomposition},
we can show that  the subspace
$\T_{g, c_1, \cdots, c_k}\cap (\T_g-S)$ associated with it is a smooth submanifold, 
and  $\T_g-S$ admits a $\Mod_g$-equivariant decomposition into disjoint
submanifolds $\T_{g, c_1, \cdots, c_k}\cap (\T_g-S)$ as  in the case of $\T_g-\T_g(\varepsilon)$.  

We also note that as in Lemma \ref{bounded}
the disjointness of the simple closed curves $c_1, \cdots,
c_k$ implies that $k\leq 3g-3$,
and the vector field $V_{c_1, \cdots, c_k}=\nabla \ell^{1/2}$
is defined on $\T_{g, c_1, \cdots, c_k}\cap (\T_g-S)$ 
and is continuous and its norm is bounded away from 0 and from above.
To prove this, we note that the norm of $\nabla \ell^{1/2}$ is  uniformly
bounded away from zero and the above when  $\ell \in (0, a]$, where $a$ is any positive
constant. The condition $\ell\in (0, a]$ for some $a>0$ is satisfied
since $\ell$ is the systole of the hyperbolic surface.

Then the same proof of Theorem \ref{can-def} works
by replacing $\T_g(\varepsilon)$ by $S$,  and Theorem \ref{spine} can be proved.
\end{proof}

 One natural question is whether the spine $S$ in Theorem \ref{spine} can be further deformation retracted to a subspace of smaller dimension.
If $S$ were a smooth manifold, then the above flow might be continued. 
In general $S$ should be a singular subspace.

Define two subspaces of $S$ by
$$S'=\{(\Sigma_g, [\varphi])\in S\mid \text{\  there are exactly two systoles \ } \gamma_1, \gamma_2 \},\  \text{ and\ \  } S''=S-S'.$$
It is clear that each hyperbolic surface in $S''$ {\em contains at least three systoles},
and hence $S''$ is a real analytic subspace of $\T_g$ of codimension at least 2.

Next we outline arguments  from Wolpert and Parlier which prove the next result. 

\begin{prop}\label{higher}
The subapce $S'$ is a smooth submanifold, and  $\nabla \ell$ is a nowhere vanishing vector
field on $S'$ such that its flow defines a deformation retraction of $S$ to $S''$.
Therefore, $S''$ is an equivariant deformation retract of $\T_g$ with codimension at least 2. 
\end{prop}
\begin{proof}
Briefly,  we note that
for each hyperbolic surface in $S'$, the two systoles $\gamma_1, \gamma_2$ 
intersect at one point. Using the fact that the difference of gradients of the length functions
of two geodesics intersecting at a single point is
never zero, in particular $\nabla \ell(\gamma_1) -  
\nabla \ell(\gamma_2)\neq 0$, 
 we conclude that $S'$ is a smooth submanifold of $\T_g$.
(See \cite[Lemma 4 in \S 8]{map}. Briefly, Thurston stretch map allows one to increase the length $\ell(\gamma_1)$
at a strictly greater rate than for the length $\ell(\gamma_2)$ in a suitable direction, since the maximal stretched set is a geodesic
lamination and can be chosen to be a complete geodesic lamination which contains $\gamma_1$ and hence not $\gamma_2$.) 

Let $\ell=\ell(\gamma_1)=\ell(\gamma_2)$ be the systole function on 
$S'$. Let $\nabla \ell$ be the gradient of $\ell$ on $S'$ with respect to the restriction
of the Weil-Petersson metric of $\T_g$.  If $\nabla \ell$ is not zero,
then $\nabla \ell$ is the direction along which both $\ell(\gamma_1)$ and  $\ell(\gamma_2)$ are increased at the maximal rate while the equality 
$\ell(\gamma_1)=\ell(\gamma_2)$ is preserved. 

Since $\gamma_1, \gamma_2$ are two intersecting systoles, 
it can  be shown that $\gamma_1, \gamma_2$  do not fill $\Sigma_g$. 
For example, when $g=2$, then the complement $\Sigma_g-\gamma_1 -\gamma_2$ is  a one holed torus. For $g\geq 2$, the complement
$\Sigma_g-\gamma_1 -\gamma_2$ is a genus $(g-1)$-surface
with one boundary component.

Let $\delta$ be a  simple closed geodesic which is disjoint from $\gamma_1, \gamma_2$.
By \cite[Lemma 3.12]{wo3}, a deformation in $\T_g$ along 
the direction of $\nabla \ell (\delta)$
will increase both $\ell(\gamma_1)$ and  $\ell(\gamma_2)$.
This implies that $\nabla \ell$ is nonzero.
Then the proof of Theorem \ref{spine}  (or Theorem \ref{can-def}) can be repeated to show
that the flow  of  $\nabla \ell$ defines
a deformation retraction of $S$ to $S''$. 
Therefore $\T_g$ admits an equivariant deformation retract $S''$ of codimension at least 2. 
\end{proof}

It seems very difficult that this deformation retract  $S''$ can
be pushed further to construct a spine of $\T_g$ with higher codimension. 
For example, it is not clear whether the subspace $S'''$ of $S''$ consisting of hyperbolic 
surfaces with {\em exact three systoles} is a smooth submanifold of $\T_g$. 
If yes, then $S''$ can be deformed as above to the subspace $S''-S'''$ of higher codimension, which {\em contains surfaces with
at least 4 systoles.}
We note that when $g=2$, $S''-S'''$ is of the optimal dimension, i.e., the virtual cohomological dimension
of $\Mod_g$, which is equal to 3.

\begin{rem}\label{thurston}
{\em The results of \cite{th} can be summarized as follows. 
\begin{enumerate}
\item Let $P$ be the subspace of $\T_g$  consisting of hyperbolic surfaces whose systoles fill the surfaces.
This is the spine proposed in \cite{th}.
\item  $P$ is a real analytic subspace of $\T_g$ and admits a triangulation, and hence $P$ is a deformation 
retraction of a regular neighborhood.
\item Thurston constructed an isotopy
$\phi_t$, $t>0$, such that for any neighborhood of $P$ and any compact subset $K$ of $\T_g$, 
there exists a $t$ for which $\phi_t(K)$ is contained in the neighborhood of $P$.
\end{enumerate}

In constructing the isotopy $\phi_t$, the key result is  \cite[0.1. Proposition: expanding subsets]{th}:
{\em Let $\Gamma$ be any collection of simple closed curves on a surface which do not fill the surface.
Then there are tangent vectors to Teichm\"uller space which simultaneously increase the lengths of the geodesics representing curves in $\ga$.}

According to \cite{th}, this is ``The only slightly original observation concerning the geometry of surfaces"
in the paper.
This was proved as follows:
(1) First cutting the hyperbolic surface along  geodesics in $\Gamma$ to obtain 
hyperbolic surfaces with boundary.  (2) Extend the surfaces to complete hyperbolic surfaces of infinite area.
(3) For any geodesic in the completed surface, cut the surface along it and glue in a strip.
Use the new expanded surface to obtain an expanded surface of the original surface. 
After this operation on several geodesics, the lengths of geodesics in $\ga$ have all been increased,
due to the assumption that $\ga$ does not fill the surface. 
Besides \cite{th}, a description of this is also given in \cite[pp. 173-174]{ha1}. 

This result was used to construct  vector fields that flow points of $\T_g$ into regular neighborhoods $P_\varepsilon$ of $P$,
which is defined to be the subset of $\T_g$ consisting of hyperbolic surfaces such that the set of simple closed geodesics
whose lengths are within $\varepsilon$ of the shortest length fill the surface.
More specifically, 
\begin{enumerate}
\item For every collection $\ga$ of simple closed geodesics which do not fill a hyperbolic surface, choose
a local  vector field along which the lengths of the geodesics in $\ga$ are all increased.
\item For every small  positive  constant $\varepsilon$,
construct a covering of $\T_g$ parametrized by collections $\ga$ of simple closed geodesics and define
a vector field on each such open subset.
If $\ga$ does not fill, use the vector field constructed in (1); if $\ga$ does, take the zero vector field.
\item Use a partition of unity defined via lengths of geodesics subordinate to the covering in (2),
and construct a global vector $X_\varepsilon$ on  $\T_g$ using the local vector fields in (2).
This vector field is zero on $P_\varepsilon$ and does not vanish on the complement $P_{B \varepsilon}$,
 where $B$ is a constant  greater than 1 and depending only on $g$.
\item Use the flow defined by the vector field $X_\varepsilon$ to deform points of $\T_g$ into a
neighborhood $P_{B\varepsilon}$ of $P$. 

\end{enumerate}
 
There seems to be some  problems with the results in \cite{th}.
The first serious one is that the vector field $X_\varepsilon$ defined in (3) may not deformation retract all the complement
$\T_g - P_{ B \varepsilon}$ into $P_{ B \varepsilon}$ in a uniform time, i.e., Step (4) might pose a problem.
Certainly there is no problem to deform any compact subset $K$ of $T_g -P_{ B \varepsilon}$
 into $P_{ B \varepsilon}$ in a fixed time, but we need to
deform the whole space.
Consider the example of the closed unit disk $D$ in $\R^2$ and a vector field $V(x)=f(x) e_1$ on $D$,
where $e_1=(1, 0)$ and $f(x)$ is a nonnegative function on $D$ which vanishes only on the unit circle.
Clearly, for every point $p$  in the interior of $D$,
the flow of $V$ will deform $p$ into any given small neighborhood of the unit circle at a finite time.
The same thing holds for every compact subset $K$ of the interior of $D$. 
But we know that the unit disk $D$ cannot be deformation retracted to the unit circle. 
This implies that based on the properties of the vector field $X_\varepsilon$ in (3),
it is {\em not necessarily true} that the flow of $X_\varepsilon$  will deform $\T_g$ into the
neighborhood $P_{B\varepsilon}$. If it can be shown that points in $P_{B\varepsilon}$ cannot be flowed
out of $P_{B\varepsilon}$, then it will be fine and Step (4) is valid. 
In summary, for this method in \cite{th} to succeed,  we need  vector fields whose flows increase 
the number of geodesics whose lengths are close to the systole of the surface, 
rather than only increasing their lengths simultaneously.

The second, nonserious problem is that when the above result is applied to these nonfilling systoles,
 their lengths are indeed increased, but it is not clear
if they have the same length.  This is the reason that the flow can only deform points of $\T_g$
into a small regular neighborhood of $P$.
For example, assume that  $\Gamma$ consists of systoles
$\gamma_1, \cdots, \gamma_k$, and the next shortest geodesic is $\gamma_{k+1}$. 
Assume further that $\Gamma$ does not fill, but $\Gamma\cup \{\gamma_{k+1}\}$ does fill. 
Then it is not clear if the above deformation will reach $P$. 
There are two alternatives to solve the second problem:
\begin{enumerate}
\item Using triangulations of $P$ and a regular neighborhood of $P$, 
the regular neighborhood of $P$ can be deformed into  $P$.
This is not canonical but depends on  the triangulations of $P$ and its
regular neighborhood.
\item Take a sequence of $\varepsilon_i$  with $\lim_{i\to \infty} \varepsilon_i=0$
 and compose their associated deformations of $\T_g$ into small
neighborhoods of $P_{B \varepsilon_i}$. 
In the limit, the deformation will reach $P$
under the assumption that the deformation retraction of $\T_g$ to $P_{B \varepsilon_i}$ works.
\end{enumerate}

It might be helpful to note that \cite[Lemma 3.12]{wo3}
implies the following result:
{\em If a collection of systoles $\gamma_1, \cdots, \gamma_k$ of a marked
hyperbolic surface $\Sigma_g$ is not filling, then a deformation
in $\T_g$ along the direction of
$\nabla \ell(\delta)$, the gradient of the length of a disjoint simple closed
geodesic $\delta$ of $\Sigma_g$,  increases simultaneously
lengths of all the systoles $\gamma_1, \cdots, \gamma_k$.}

This gives a different proof of \cite[0.1. Proposition]{th} recalled above.
But this probably does not overcome the  problem as pointed out above: it does not obviously  lead a good direction to
increase the lengths of the systoles $\gamma_1, \cdots, \gamma_k$ and also to  make then closer in some sense. 

}
\end{rem}




\end{document}